\documentclass[11pt,oneside]{amsart}

\usepackage{amsmath}
\usepackage{geometry}
\usepackage{hyperref}
\usepackage{graphicx}

\theoremstyle{theorem}
\newtheorem{theorem}{Theorem}

\newtheorem{lemma}{Lemma}

\newcounter{abc}
\newtheorem{thmPrime}[abc]{Theorem}

\theoremstyle{definition}

\theoremstyle{remark}

\title[Local rigidity of convex hypersurfaces]{Local rigidity of convex hypersurfaces in spaces of constant curvature}
\author{Alexander A. Borisenko}

\date{}

\address{B. Verkin Institute for Low Temperature Physics and 
Engineering of the National  Academy of Sciences of Ukraine,
47 Nauky Ave., Kharkiv, 61103, Ukraine;}
\address{Brown University -- ICERM, USA}
\email{aborisenk@gmail.com}

\begin{document}

\begin{abstract}
In this paper, we prove a local rigidity of convex hypersurfaces in the spaces of constant curvature of dimension $n\ge4$. Namely, we show that two convex isometric hypersurfaces are congruent locally around their corresponding under the isometry points of strict convexity. This result extends the result of E.\,P.~Senkin, who showed such rigidity under the additional assumption of $C^1$-smoothness of the hypersurfaces. 

\medskip	

\noindent
\textbf{Keywords:} rigidity, convex hypersurface, space of constant curvature.

\medskip
\noindent
\textbf{MSC2020: } 52A10, 52A55, 51M10, 53C22.
\end{abstract}

\maketitle

\section{Introduction}

In 1950, A.\,V.~Pogorelov proved the unique determination of closed convex surfaces in the Euclidean space $E^3$ \cite{Pog}. For this result, no other conditions on the closed surface except convexity were assumed. Under additional regularity assumptions, the theorem had been proven earlier by S.~Cohn-Vossen in 1924 \cite{CV} and by G.~Herglotz in 1943 \cite{Her}. A.\,V.~Pogorelov also established the unique determination of closed convex surfaces in the spherical space $S^3$. Using the result of Pogorelov and an idea of Senkin, A.\,D.~Milka \cite{Mil} proved the same result for closed convex surfaces in the hyperbolic (Lobachevsky) space $\mathbb H^3$. E.\,P.~Senkin \cite{Sen1}, \cite{Sen3} generalized Pogorelov's theorem to closed convex hypersurfaces in the Euclidean space $E^{n+1}$ ($n \ge 3$) under the additional assumption of $C^1$ regularity of the hypersurface. In \cite{Bor}, this theorem was established without any additional assumptions. 

E.\,P.~Senkin proved the following result on \emph{local} unique determination of $C^1$-smooth convex hypersurfaces:

\begin{theorem}[\cite{Sen1}, \cite{Sen3}]
\label{Thm:Sen2a}
Let $F_1 \subset E^{n+1}$, $n\ge 3$, be a $C^1$-smooth convex hypersurface and $P_1 \in F_1$ be a point of strict convexity. Suppose $F_2 \subset E^{n+1}$ is a $C^1$-smooth convex hypersurface that is isometric to $F_1$, and let $P_2 \in F_2$ be the point corresponding to $P_1$ under this isometry. Then $F_1$ and $F_2$ are congruent in sufficiently small neighborhoods of $P_1$ and $P_2$.   \qed
\end{theorem}

In the proof of this theorem, the assumption of smoothness was used to prove that the isometric combination of $F_1$ and $F_2$ in a neighborhood of $P_1$ and $P_2$ is a convex hypersurface. 

Since we proved this without the smoothness assumption, then Theorem~\ref{Thm:Sen2a} is now proven without this extra assumption.

\begin{thmPrime}
\label{Thm:Prime1}
Let $F_1 \subset E^{n+1}$, $n \ge 3$, be a convex hypersurface and $P_1 \in F_1$ be a point of strict convexity. If $F_2 \subset E^{n+1}$ is a convex hypersurface that is isometric to $F_1$ and $P_2 \in F_2$ is the point corresponding to $P_1$ under the isometry, then $F_1$ and $F_2$ are congruent in sufficiently small neighborhoods of $P_1$ and $P_2$.    
\end{thmPrime}

Another result by Senkin is the following theorem:

\begin{theorem}[\cite{Sen2}, \cite{Sen3}]
\label{Thm:Sen2b}
Let $F_1, F_2\subset E^{n+1}$, $n\ge 3$, be $C^1$-smooth convex isometric hypersurfaces. Let $P_1 \in F_1$ be a point that does not lie in any flat region of $F_1$ of dimension $n$, $(n-1)$, $(n-2)$, and let $P_2 \in F_2$ be the point corresponding to $P_1$ under the isometry.  Then $F_1$ and $F_2$ are congruent in sufficiently small neighborhoods of $P_1$ and $P_2$.   \qed
\end{theorem}

Again, in the proof of this theorem, the assumption of smoothness was used to prove that the isometric combination of $F_1$ and $F_2$ in a neighborhood of $P_1$ and $P_2$ is a convex hypersurface. Since we proved that we can position $F_1$ and $F_2$ in such a way that their isometric combination in a neighborhood of the points $P_1$ and $P_2$ is a convex hypersurface, then Theorem~\ref{Thm:Sen2b} is true without the smoothness assumption.

\begin{thmPrime}
\label{Thm:Prime2}
Let $F_1, F_2 \subset E^{n+1}$, $n \ge 3$, be isometric convex hypersurfaces, and let $P_1 \in F_1$ be a point that does not lie in any flat region of $F_1$ of dimension $n$, $(n-1)$, $(n-2)$, and $P_2 \in F_2$ be the point corresponding to $P_1$ under the isometry. Then sufficiently small neighborhoods of $P_1$ and $P_2$ in $F_1$ and $F_2$ are congruent.
\end{thmPrime}

Now we will prove that a pair of convex isometric hypersurfaces $F_1$ and $F_2$ can be positioned by aligning a pair of corresponding under the isometry points $P_1 \in F_1$ and $P_2 \in F_2$ in such a way that in a neighborhood of $P_1, P_2$ the isometric combination of the hypersurfaces is again a convex hypersurface.

Let $F_1, F_2 \subset E^2$ be a pair of simple closed convex positively oriented curves of the same length and with a pair of marked points $P_1 \in F_1, P_2 \in F_2$. By a parallel translation of the curves, we can align the points $P_1$ and $P_2$. By a rotation around the point $P := P_1 = P_2$, we can assume that $F_1$ and $F_2$ have a common supporting line at $P$ (so that $F_1$ and $F_2$ lie in the same half-plane with respect to this line). In this way, the regions bounded by the curves $F_1$ and $F_2$ have a non-empty intersection. Let $P_1(s), P_2(s)$ be the points on $F_1, F_2$ such that the lengths of the arcs $PP_1(s)$, $PP_2(s)$ along $F_1, F_2$ are both equal to $s$. Denote by $r_1(s), r_2(s)$ the position vectors of $P_1(s), P_2(s)$. Then the \emph{isometric combination} of the curves $F_1$, $F_2$ is the set $F$ given by the radius-vector
\[
r(s) = r_1(s) + r_2(s).
\]

We denote by $P(s)$ the isometric combination of $P_1(s)$ and $P_2(s)$.

\begin{lemma}
\label{Lem:1}
Let $F_1, F_2 \subset E^2$ be a pair of simple closed convex positively oriented polygons
of the same length and with a pair of marked points $P_1 \in F_1$ and $P_2 \in F_2$.
If for every $s$, the angle between the right half-tangents at the points $P_1(s), P_2(s)$ is less than $\pi$, then $F$ is a simple closed convex polygon.
\end{lemma}

\begin{proof}
Only the following three cases are possible:

\begin{enumerate}
\item
The points $P_1(s), P_2(s)$ lie in the interior of the sides of the polygons. Then the isometric combination of the neighborhoods of points will be a segment containing the point $P(s)$.

\item
The point $P_1(s)$ lies in the interior of the side of $F_1$ and the point $P_2(s)$ is the vertex of $F_2$. Recall that the convex polygons in question have the same orientation. Let $\alpha_i$, respectively $\delta_i$ be the angle between the right, respectively left semi-tangents at the points $P_1(s)$ and $P_2(s)$, $\beta^1_i = \pi$ and $\beta^2_i < \pi$ be the angles of the polygons at the vertices $P_1(s), P_2(s)$, and let $\gamma_i$ be the angle between the right semi-tangent at $P_2(s)$ and the left semi-tangent at $P_1(s)$.  By construction, $P(s)$ is the vertex with angle $\beta_i := \alpha_i / 2 + \delta_i / 2 + \gamma_i$. But
\[
\delta_i + \gamma_i = \beta_i^1 = \pi, \quad \alpha_i + \gamma_i = \beta_i^2 < \pi.
\]
This gives that
\[
\beta_i = \frac{\alpha_i}{2} + \frac{\beta_i}{2} + \gamma_i = \frac{\beta^1_i + \beta^2_i}{2} < \pi.
\]
Observe that we have chosen the orientation of $F_1$ in such a way that $F_1$ lies in the upper half-plane with respect to the side of $F_1$ containing $P_1(s)$. It is therefore impossible that the both sides of $F_2$ at the vertex $P_2(s)$ lie in the lower half-plane (otherwise, the polygons $F_1$ and $F_2$ do not intersect, which is a contradiction to our choice above). 

\item
The points $P_1(s), P_2(s)$ are the vertices of their respective polygons. Let us make a parallel translation so that the vertices $P_1(s), P_2(s)$ coincide. Without loss of generality, we can assume that the distance of the vertex $P_1(s)$ to the supporting line at $P_1 = P_2$ is not greater than the distance of the vertex $P_2(s)$. Then the parallel translation $T$ of the point $P_1(s)$ to the point $P_2(s)$ will move the polygon $F_1$ to the polygon $TF_1$ that also intersects $F_2$. The case when the solid angles at the vertices $TP_1(s)$ and $P_2(s)$ do not intersect is impossible because otherwise the intersection of polygons is empty. We obtain that the angle at the vertex $P(s)$ of the isometric combination is $\beta_i = (\beta_i^1 + \beta_i^2)/2 < \pi$. 
\end{enumerate}

For the convex polygons $F_1, F_2$, we have that
\[
\sum (\pi - \beta_i^1) = 2\pi, \quad \sum (\pi - \beta_i^2) = 2\pi.
\]
Thus
\[
\sum (\pi - \beta_i) = \beta_i = \sum \frac{\pi - \beta_i^1}{2} + \sum \frac{\pi - \beta_i^2}{2} = 2\pi. 
\]
Since $(\pi - \beta_i) \ge 0$, the closed polygonal line $F$ will be a closed convex polygon. We obtain that $F$ is the isometric combination of $F_1$ and $F_2$.  
\end{proof}

This result generalizes easily to the case of arbitrary convex curves.

\begin{lemma}
\label{Lem:2}
Let $F_1, F_2$ be closed convex curves of the same length $s_0$ and with the same orientation. Let $P_1 \in F_1$, $P_2 \in F_2$ be a pair of marked points with respect to which the length of arcs along the curve is measured. Assume that $P_1 = P_2$ and at this point the curves $F_1, F_2$ have a common supporting line. Further assume that the regions bounded by $F_1, F_2$ intersect. If at the corresponding with respect to the arc length points $P_1(s), P_2(s)$ the angle between the right semi-tangents is less than $\pi$, then the isometric combination of $F_1, F_2$ is a closed convex curve.  
\end{lemma}

\begin{proof}
Let $F_1^n ,F_2^n$ be a sequence of convex polygons inscribed into $F_1, F_2$ such that
\[
\lim_{n \to \infty} F_1^n = F_1, \quad \lim_{n \to \infty} F_2^n = F_2.
\]
Then the lengths of $F_1^n, F_2^n$ tend to $s_0$. Let us make a dilatation at a fixed point $P_1 = P_2 = P$ so that the length of image polygon $\tilde F_1^n$ coincides with the length of $F_2^n$. For sufficiently large $n$, the angles between the tangents to $\tilde F_1^n, F_2^n$ at the corresponding under the isometry points is less than $\pi$. By Lemma~\ref{Lem:1}, the isometric combination of the polygons $\tilde F_1^n$, $F_2^n$ is a closed convex polygon $F^n$. The limit of the polygons $F^n$ is the isometric combination of the curves $F_1, F_2$. This will be a convex curve $F$.   
\end{proof}

\begin{lemma}
\label{Lem:3}
Let $F_1, F_2$ be two closed nondegenerate convex curves with marked points $P_1, P_2$. Suppose that $F_1, F_2$ have the same orientation and the same length $s_0$. Then there exists an isometry of the plane that moves these curves so that the angles between the right semi-tangents at the points $P_1(s), P_2(s)$ is less than $\pi$, where $s$ is the arc length measured starting from $P_1=P_1(0), P_2=P_2(0)$. 
\end{lemma}

\begin{proof}
Let us align the points $P_1(0)$ and $P_2(0)$ using an appropriate isometry of the plane in such a way that these points coincide and that their right semi-tangents at these points also coincide. 
If the curve $F_1$ is strictly convex, then it is possible to choose as the point $P_1(0)$ any point of $F_1$. If on the curve $F_1$ there exists a line segment, then the beginning of the segment we choose as the point $P_1(0)$. The orientation on the segment is the same as the orientation on the curve $F_1$.
If for every $s \in (0, s_0]$, the angle between the right semi-tangent at $P_1(s)$ and $P_2(s)$ is $<\pi$, then we are done. So suppose this is not the case. Then there is a pair of points $P_1(\sigma_0)$, $P_2(\sigma_0)$ such that the right semi-tangents at those points have the opposite direction. For definiteness, we can assume that the turning of the arc $P_1(0)P_1(\sigma_0)$ is bigger than the turning of the arc $P_2(0)P_2(\sigma_0)$. Then
\[
\tau_1(\sigma_0) = \tau_2(\sigma_0) + \pi,
\] 
where for $i \in \{1,2\}$, $\tau_i(s)$ is the turning of the arc of the curve $F_i$ from the point $P_i(0)$ to the point $P_i(s)$. Observe that $\tau_2(\sigma_0) < \pi$ because the complete turning of a closed convex curve is equal to $2\pi$ and the point $P_1(\sigma_0)$ does not coincide with the point $P_1(0)$. Furthermore, both $\tau_1(s)$ and $\tau_2(s)$ are non-negative and non-decreasing functions.

For each $s \in [0, s_0]$, let us align the points $P_1(s)$ and $P_2(s)$ in the same way as we did above for $P_1(0)$ and $P_2(0)$, and assume that we have the same situation with the oppositely oriented right semi-tangents. That is, for the points $P_1(s)$, $P_2(s)$ we would find the points $P_1\big(f(s)\big)$, $P_2\big(f(s)\big)$ such that
\[
\omega_1(s) = \omega_2(s) + \pi, \quad 0 < \omega_2(s) < \pi,
\]  
where $\omega_1(s)$, $\omega_2(s)$ are the turnings of the arcs $P_1(s) P_1\big(f(s)\big)$, $P_2(s) P_2\big(f(s)\big)$.

Let us pick such $s$ so that
 \begin{equation}\label{turning}
    2\pi > \tau_2(s) + \omega_2\big(f(s)\big) = \tau_2\big(f(s)\big) > \pi.
 \end{equation}
 
Then $\tau_1\big(f(s)\big) > 2\pi$. But since $\tau_2\big(f(s)\big) < 2\pi$, then $f(s) < s_0$ and $\tau_1\big(f(s)\big) < 2\pi$. This is a contradiction. Therefore, there exists $\sigma_0 < s_0$ such that the isometry that maps $P_1(\sigma_0)$ to $P_2(\sigma_0)$, and the right semi-tangent at $P_1(\sigma_0)$ to the right semi-tangent at $P_2(\sigma_0)$ is the required isometry.

Let us explain it in detail. Let us take on $F_2$ the point $P_2(s)$ such that the turning 
$\tau(P_2(0)P_2(s))=\pi$. It is possible the possible cases: \\
1) such point $P_2(s)$ exists. If $\omega_2(s) >0$, then the inequality (\ref{turning}) is true; \\
2) Let us assume that $\omega_2(s)=0$. Then 
$$ 2\pi > \tau_2(s)+\omega_2\big(f(s)\big)=\pi \quad \text{and} \quad \tau_1\big(f(s)\big)=2\pi.$$
But $\tau_2\big(f(s)\big)= \pi$ and it follows that $f(s) < s_0$. The point $P_1(f(s))$ does not coincide
with  the point $P_1\big(f(s)\big)$, and the segment $P_1\big(f(s)\big)P_1(0)$ belongs to the curve $F_1$. 
If the point $P_1(0)$ is a corner point, then turning of the curve $F_1 >2\pi$. 
It is impossible for a closed convex curve. 
If the point $P_1(0)$ is a smooth point, then we have a contradiction with the choice of the point $P_1(0)$. \\
3) The turning of the arc $P_2(0)P_2(s)$ is greater than $\pi$ in the case when the point $P_2(s)$ is a corner point. The turning of the arc $P_2(s-\varepsilon)P_2(s)$ goes to $\pi - \alpha$ if $\varepsilon \to 0$, $\alpha$ is the angle at the point $P_2(s)$. 
Suppose that on the arc $P_2(s-\varepsilon)P_2(f(s-\varepsilon))$ , $f(s-\varepsilon) >s$ the following holds
$$\omega_1(s-\varepsilon) = \omega_2(s-\varepsilon) + \pi.$$
Then on the arc $P_1(s-\varepsilon)P_1(s)$ there exists a point with parallel support lines with the 
distance $\leq \varepsilon$. The curves $F_1$, $F_2$ are nongenerate and for $\varepsilon < c$, 
where $c >0$ is a constant, it is impossible- 
Then for $\varepsilon < c$ it is true strong inequalities
$$ 2\pi > \tau_2(s-\varepsilon) +  \omega_2(s-\varepsilon) =   \tau_2\big(f(s-\varepsilon)\big) > \pi,  $$
and we obtain contradiction as in the case 1).

\end{proof}

Lemmas \ref{Lem:1}, \ref{Lem:2}, \ref{Lem:3} yield the following theorem.

\begin{theorem}
\label{Thm:3}
Let $F_1, F_2$ be two simple closed convex curves on the plane oriented in the same way and of the same length. Then there exists an orientation-preserving isometry of the plane that maps $F_1, F_2$ to a pair of curves whose isometric combination is also a simple closed convex curve. \qed
\end{theorem}

Theorem~\ref{Thm:3} gives

\begin{theorem}
\label{Thm:4}
Let $F_1, F_2$ be convex isometric cones in $E^3$. Then there exists an isometry of $E^3$ that maps those cones to the cones whose isometric combination is a strictly convex cone. \qed
\end{theorem}

\section{Convex isometric cones in $E^3$}

Let $K_1, K_2$ be convex isometric cones. We align those cones so that their vertices are at the origin, and such that they share a common strictly supporting plane. We further suppose that they are given by the position vectors
\[
R_1 = t \cdot r_1(s), \quad R_2 = t \cdot r_2(s), \quad \langle r_1, r_1 \rangle  = \langle r_2, r_2 \rangle = 1.
\]
The position vectors $r_1 = r_1(s)$, $r_2 = r_2(s)$ encode the curves $M_1, M_2$ at the intersection of the cones with the unit sphere $S^2$, $x_0^2+x_1^2+x_2^2=1$, where $x_0, x_1, x_2$ are the Cartesian coordinates in $E^3$ chosen so that the curves have $x_0 > 0$, and $s$ is the arc-length parameter measured along those curves starting at points $P_1$, respectively $P_2$, that correspond each other under the cone isometry. The curves $M_1, M_2$ are convex curves in $S^2$ of the same length.

Using the Pogorelov transformation, let us map the curves $M_1, M_2$ to the $(x_1, x_2)$-plane:
\[
\tilde r_1 = \frac{r_1 - \langle r_1, e_0 \rangle}{\langle e_0, r_1 + r_2 \rangle}, \quad \tilde r_2 = \frac{r_2 - \langle r_2, e_0 \rangle}{\langle e_0, r_1 + r_2 \rangle},
\]
where $e_0$ is the unit vector in the direction of $x_0$ axes, $\tilde r_1, \tilde r_2$ are the orthogonal projections of the vectors $r_1, r_2$ onto the $(x_1, x_2)$-plane. By a theorem of Pogorelov \cite[Ch.5, \S 3]{Pog}, the curves $\tilde M_1, \tilde M_2$ given by the position vectors $\tilde r_1, \tilde r_2$ are convex and isometric. Therefore, by Theorem~\ref{Thm:3}, the curves $\tilde M_1, \tilde M_2$ can be rotated in the $(x_1, x_2)$-plane so that their isometric combination is a closed convex curve $\tilde M$.

Observe that if we rotate the curve $\tilde M_1$, then it corresponds to the rotation of the cone $K_1$ and the corresponding convex curve $M_1$ around the $x_0$-axis. Therefore, without loss of generality, we can assume that the isometric combination of curves $\tilde M_1$ and $\tilde M_2$ is a convex curve $\tilde M$ given by the position vector $\tilde r = \tilde r_1 + \tilde r_2$. This curve corresponds to a convex curve $M_0 \subset S^2$ with the position vector
\[
r = \frac{(\tilde r_1 + \tilde r_2, 1)}{\sqrt{1 + \langle \tilde r_1 + \tilde r_2, \tilde r_1 + \tilde r_2 \rangle}}.
\]
We can express $\tilde r_1, \tilde r_2$ in another way as follows:
\[
\tilde r_1 = \frac{\bar r_1}{x_1^0 + x_2^0}, \quad \tilde r_2 = \frac{\bar r_2}{x_1^0 + x_2^0} ,
\]
where $\bar r_1, \bar r_2$ are the orthogonal projections of the position vectors $r_1, r_2$ onto the $(x_1, x_2)$-plane, and $x_1^0, x_2^0$ are the $x_0$-coordinates of $r_1$, respectively, $r_2$.

By direct computation, it follows that
\[
r = \frac{\left(\frac{\bar r_1 + \bar r_2}{x_1^0 + x_2^0}, 1\right)}{\sqrt{1 + \frac{\langle \bar r_1 + \bar r_2, \bar r_1 + \bar r_2 \rangle}{(x_1^0 + x_2^0)^2}}} = \frac{\left({\bar r_1 + \bar r_2}, {x_1^0 + x_2^0}\right)}{\sqrt{\left(x_1^0+x_2^0\right)^2 + \langle \bar r_1 + \bar r_2, \bar r_1 + \bar r_2 \rangle}} = \frac{(\bar r_1 + \bar r_2, x_1^0 + x_2^0)}{\sqrt{2(1 + \langle r_1, r_2 \rangle)}}.
\]

The cones $K_1, K_2$ are given by the position vectors
\[
R_1 = t \cdot r_1, \quad R_2 = t \cdot r_2.
\]
Their isometric combination is a cone $K$ with the position vector $R = t \cdot (r_1 + r_2)$. The curve $M$ of the intersection of this cone with the unit sphere is given by
\[
r = \frac{r_1 + r_2}{\sqrt{\langle r_1 + r_2, r_1 + r_2\rangle}} = \frac{(\bar r_1 + \bar r_2, x_1^0 + x_2^0)}{\sqrt{2(1 + \langle r_1, r_2 \rangle)}}.
\]
As we can see, the curves $M$ and $M_0$ coincide. Hence, the isometric combination of the cones $K_1$ and $K_2$ will be the convex cone $K$.

\section{Isometric combinations of dihedral angles}

If for a pair $D_1, D_2$ of dihedral angles their edges correspond under the isometry, them the isometric combination of $D_1$ and $D_2$ will be a dihedral angle $D$. If we put one dihedral angle inside the other so that their intersection with the unit sphere looks as shown in Figure~\ref{Fig:1} (left), then the angle between the corresponding under the isometry rays of dihedral angles is $< \pi - \epsilon_0$, $0 < \epsilon_0 < \pi$. 

If the edges of $D_1, D_2$ do not correspond to each other under the isometry, we can still locate the angles so that their edges coincide and that their intersection with the unit sphere has the form as shown in Figure~\ref{Fig:1}, left. In this case, however, the vertices of the digons do not correspond each other under the isometry. Nonetheless, in this case the angle between the corresponding under the isometry directions $< \pi - \epsilon_0$, $0 < \epsilon_0 < \pi$.

\begin{figure}[h]
\begin{center}
\includegraphics[width=0.5\textwidth]{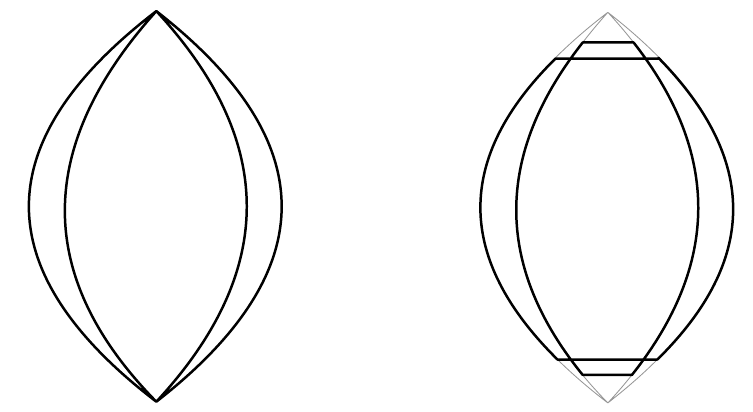}
\caption{}
\label{Fig:1}
\end{center}
\end{figure}

Let us substitute the digons with the sequence of the quadrilaterals $D_1^n, D_2^n$ of the same length (see Figure~\ref{Fig:1}, right). Their isometric combination will be a convex curve $D^n$. The limit of $D^n$'s in the sphere will be a convex curve $D$ which is the intersection with the sphere of the isometric combination of $D_1$ and $D_2$. 

We can write them as a Cartesian product of a dihedral angle in $E^3$ times $E^{n-2}$:
\[
K_1(P_1) = D_1 \times E_1^{n-2}, \quad K_1(P_2) = D_2 \times E_2^{n-2},
\]     
where $E_1^{n-2}$, $E_2^{n-2}$ coincide and correspond to each other under the isometry. The isometric combination of such cones will be the Cartesian product of the isometric combination of $D_1$ and $D_2$ times $E^{n-2}$. Let us locate the dihedral angles so that they are nested. Then on the sphere they cut out digons that are also nested. The following cases are possible:

a) the edges of the dihedral angles correspond each other under the isometry. Then their isometric combination will be a dihedral angle.

b) the edges do not correspond under the isometry. In this case, let us substitute the digons in $S^2$ with the quadrilaterals $D_1^n, D_2^n$ that are obtained by truncating the vertices of the digons so that the lengths of $D_1^n$ and $D_2^n$ are equal. By Theorem~\ref{Thm:3}, the isometric combination of polyhedral angles that correspond to these quadrilaterals will be a convex polyhedral angle $D^n$. Moreover, for the polyhedral angles $D_1, D_2$, the angle between the isometric directions $< \pi - \epsilon_0$, $\epsilon_0 > 0$. The limit of $D^n$'s will be a convex polyhedral angle $D$ that is a isometric combination of dihedral angles in $E^3$.

In this way, Theorem~\ref{Thm:4} is proven.

\section{Proof of Theorems~\ref{Thm:Prime1} and \ref{Thm:Prime2}}

Let us place the tangent cones at the points $P_1 \in F_1$, $P_2 \in F_2$ so that their isometric combination is a convex cone. This can be done indeed:

1) Suppose that the tangent cone at $P_1 \in F_1$ has a form $K(P_1) = V_1^s \times E_1^{n-s}$, $s \ge 3$. Then, as was proven in \cite{Bor}, the cone $K(P_2)$ is congruent to $K(P_1)$, i.e., there is an isometry of $E^3$ that sends one cone to the other. Hence, their isometric combination will be a cone that is homothetic to $K(P_1)$. 

2) $K(P_1) = V_1^2 \times E_1^{n-2}$. Then $K(P_2) = V_2^2 \times E_2^{n-2}$, where $E_1^{n-2}, E_2^{n-2}$ are isometric planes. We can align those plane by a proper motion, and the isometric cones $V_1^2, V_2^2$ will be in the common three dimensional space $E^3$. From Theorem~\ref{Thm:3} and the Pogorelov transformation it follows that there exists a motion in $E^3$ that moves the cones in the position so that their isometric combination is a convex cone.

3) $K(P_1) = V_1^1 \times E_1^{n-1}$ is a dihedral angle in $E^{n+1}$ (here, $V_1^1$ is a cone in $E^2$), $K(P_2) = V_2^1 \times E_2^{n-1}$. We can represent them as a Cartesian product 
\[
K(P_1) = D_1 \times E_1^{n-2}, \quad K(P_2) = D_2 \times E_2^{n-2},
\]
where $E_1^{n-2}, E_2^{n-2}$ are the Euclidean subspaces that correspond each other under the isometry, and $D_1, D_2$ are the dihedral angles in $E^3$. They can be located so that their isometric combination is a convex cone.

4) $K(P_1) = E_1^n$, $K(P_2) = E_2^n$.

5) $K(P_1) = V_1^1 \times E_1^{n-1}$, $K(P_2) = E^n$. In those cases, their isometric combination will be 4) hyperplane, 5) dihedral angle.

Consider the cases:

1). Let $P_1^n \in F_1$, $P_2^n \in F_2$ be a sequence of corresponding under the isometry points that tend to $P_1 = P_2 = P_0$. The cones $K(P_1^n)$, $K(P_2^n)$ tend to the isometric cones $K_1^0$, $K_2^0$. Since the cones $K(P_1)$ and $K(P_2)$ coincide, then the convex cones $K_1^0$ and $K_2^0$ coincide and their isometric combination is a convex cone. 

Let us show that for sufficiently large $n$, the isometric combination of the cones $K(P_1^n)$, $K(P_2^n)$ is a convex cone.

a) Suppose $K(P_1^n) = V_1^s \times E_1^{n-s}$, $s \ge 3$. Then the cones $K(P_1^n)$ and $K(P_2^n)$ are congruent and their isometric combination is a convex cone for sufficiently large $n$.

b) Suppose $K(P_1^n) = V_1^2 \times E_1^{n-2}$. Then $K(P_2^n) = V_2^2 \times E_2^{n-2}$, and the cones $K_1^0, K_2^0$ are either coinciding $V^2 \times E^{n-2}$, or $V^1 \times E^{n-1}$, or $E^n$. Suppose that on the cone $\gamma_n$, which is a isometric combination of cones $K(P_1^n)$, $K(P_2^n)$, there exists a generator $t_n$ at which the local convexity is violated. Without loss of generality, we can assume that the corresponding generators $t_n^1$ on $K(P_1^n)$ converge to $t_0$ on $K_1^0$, where $t_0 \in V_1^s$. 

Let $A_n'$ be a point on $t_n^1$ at distance $1$ from the subspace $E_1^{n-2}$, $\kappa_n'$ be a tangent dihedral angle at this point, and $\kappa_n''$ be a tangent dihedral angle at the point $A_n'' \in K(P_2^n)$ that corresponds to $A_n'$ under the isometry. Under the isometric combination of the angles $\kappa_n', \kappa_n''$ for sufficiently large $n$, we obtain a dihedral angle $\kappa_n$ that contains a ball $\omega$. The existence of the ball $\omega$ follows from the fact that $\kappa_1^0$ and $\kappa_2^0$ coincide and the limits of $K(P_1^n), K(P_2^n)$ are $K_1^0, K_2^0$. Passing through the edge of this dihedral angle we can draw a supporting plane so that the inner normal $\nu$ to this plane, if placed starting at $A_n'$, $A_n''$, directs inside the cones $K(P_1^n)$, $K(P_2^n)$. Let us connect a point $B_n'$, that is close to $A_n'$, with the geodesic $\gamma'$ on $K(P_1^n)$. Let $r_1(s)$ be the position vector of a point on the curve $\gamma'$ corresponding to the arc-length $s$ (measured from the point $A_n'$), and $r_2(s)$ be the position vector of a point on $K(P_2^n)$ corresponding under the isometry. For $s=0$, we have
\[
\frac{d}{ds}\langle r_1 + r_2, \nu \rangle \ge 0.
\]  
By a theorem of Liberman about the convexity of geodesics $\gamma'$ on $K(P_1^n)$ and the corresponding geodesic $\gamma''$ on $K(P_2^n)$, we conclude that
\[
\frac{d}{ds} \langle r_1 + r_2, \nu \rangle \ge 0 \quad \text{ for all }s \text{ along }\gamma'.
\]

By this inequality, all points of the cone $\gamma_n$ that are close to the image of the point $A_n'$ are located on one side with respect to the supporting hyperplane with the inner normal $\nu$. This yields the convexity of the cone. Here, for completeness, we have outlined the argument due to Pogorelov about the convexity of the isometric combination of isometric cones in $E^3$ \cite{Pog}. 

2) $K(P_1) = V_1^2 \times E_1^{n-2}$, $K(P_2) = V_2^2 \times E_2^{n-2}$. We can align the vertices $P_1 = P_2 = P_0$ and place $F_1$ and $F_2$ so that the isometric combination of the cones $K(P_1)$, $K(P_2)$ is a convex cone. Let $P_1^n \in F_1$, $P_2^n \in F_2$ be a sequence of points that correspond each other under the isometry and $P_1^n \to P_0$, $P_2^n \to P_0$ as $n \to \infty$.
Further, if $K(P_1^n)$, $K_2(P_2^n)$ are the tangent cones, $\lim_{n\to \infty} K(P_1^n) = K_1^0$, $\lim_{n\to \infty} K(P_2^n) = K_2^0$.

a) Let $K_1^0 = V_1^2 \times E^{n-2}$. Then $K_2^0 = V_2^2 \times E^{n-2}$. Then $K(P_1) = K_1^0$, $K(P_2) = K_2^0$, and the isometric combination of $K_1^0, K_2^0$ is a convex cone. Furthermore, there exists a ball $\omega$ that lies inside the cones $K_1^0, K_2^0$. The rest of the proof goes as in item 1) above.

b) $K_1^0$, $K_2^0$ are dihedral angles. Since those dihedral angles are supporting to the cones $K(P_1)$, $K(P_2)$, and their isometric combination is a convex cone, then there exists a ball $\omega$ inside these dihedral angles. The rest of the proof follows the proof in item 1) above.

Items 3)--5) are proven analogously.  

In this way, we showed that in a neighborhood of the point $P_0 = P_1 = P_2$ the hypersurfaces $F_1$ and $F_2$ can be moved by a rigid motion in such a way so that their isometric combination in a neighborhood of $P_0$ be a convex hypersurface. The rest of the proofs of Theorems~\ref{Thm:Prime1}, \ref{Thm:Prime2} follows the proof of Theorems \ref{Thm:Sen2a}, \ref{Thm:Sen2b} due to Senkin \cite{Sen, Sen2}. We briefly outline it here.

\textit{Proof of Theorem~\ref{Thm:Prime1}}. Let $F_1$ and $F_2$ be convex isometric hypersurfaces, $P_1 \in F_1$ be a point of strict convexity of the hypersurface $F_1$, and $P_2 \in F_2$ be the corresponding under the isometry point. By Theorem~\ref{Thm:3}, using the isometry of the ambient space, we can align the points $P_1$ and $P_2$ so that their isometric combination given by the position vector $r = (r_1 + r_2)/2$ is a hypersurface $\Phi$ convex in a neighborhood of the point $P_0 = P_1 = P_2$, while the field $\tau = r_1 - r_2$ is an infinitesimal bending field of $\Phi$.  

It follows that almost everywhere 
\[
\langle dr, d\tau \rangle = 0.
\] 

Let us move the supporting hyperplane at $P_0$ by parallel translations so that it cut away a cap $\omega$ of the hypersurface $\Phi$. Senkin \cite{Sen} showed that the cap $\omega$ is rigid away from the hyperflat parts, i.e., the field $\tau$ is trivial away from such hyperflat regions.

Such hyperflat regions on $\Phi$ can happen only as isometric combinations of hyperflat regions on $F_1$ and $F_2$. Therefore, since the regions on $F_1$, $F_2$ that correspond to $\Phi$, coincide away from the flat regions, then they must coincide.

\textit{Proof of Theorem~\ref{Thm:Prime2}} coincides with the proof of Theorem~\ref{Thm:Sen2b} due to Senkin \cite{Sen2}, \cite{Sen3} after we established Theorem~\ref{Thm:4}.

As a corollary to Theorem~\ref{Thm:Prime2}, we obtain the following result: \textit{Let $F_1, F_2$ be isometric convex hypersurfaces. If at a point $P_1 \in F_1$ the tangent cone has a form $K(P_1) = V_1^s \times E^{n-s}$, $s \ge 3$, then there exists a neighborhood of the point $P_1$ that is congruent to the corresponding isometric neighborhood on the hypersurface $F_2$.}

A.~D.~Alexandrov proved that at almost every point of a convex hypersurface there exists a second fundamental form. If the rank of the second fundamental form at such a point is $\ge 3$, then at this point there exists a neighborhood on the hypersurface that is rigid up to proper isometry.

Theorems~\ref{Thm:Prime1}, \ref{Thm:Prime2} are true in the spherical and the hyperbolic spaces. To see that, we use the Pogorelov transform that sends the isometric convex hypersurfaces in those spaces to the isometric convex hypersurfaces in $E^{n+1}$ satisfying the assumptions of Euclidean Theorems~\ref{Thm:Prime1}, \ref{Thm:Prime2}. Now using the inverse transformation, we obtain the congruent hypersurfaces in spaces of constant curvature.  



\subsection*{Acknowledgments.} 
The author was partially supported by the Simons Foundation Award 507536. He is grateful to ICERM and the Mathematical Department of Brown University for hospitality during his stay. The author also would like to thank Kostiantyn Drach and Darya Sukhorebska  for the  help in the preparation of the manuscript and several useful remarks.



\end{document}